\documentclass[lettersize,journal]{IEEEtran}
\usepackage{amsfonts}
\usepackage{algorithm}
\usepackage{array}
\usepackage[caption=false,font=normalsize,labelfont=sf,textfont=sf]{subfig}
\usepackage{textcomp}
\usepackage{stfloats}
\usepackage{url}
\usepackage{verbatim}
\usepackage{cite}
\usepackage{microtype}
\usepackage{graphicx}
\usepackage{booktabs}
\usepackage{hyperref}
\usepackage[textsize=tiny]{todonotes}
\usepackage[bibliography=common]{apxproof}

\makeatletter
\renewcommand{\appendix}{%
  \par
  \setcounter{section}{0}%
  \renewcommand{\thesection}{\Alph{section}}%
}
\makeatother

\usepackage{amsmath}
\usepackage{amssymb}
\usepackage{mathtools}
\usepackage{amsthm}
\usepackage[capitalize,noabbrev]{cleveref}
\usepackage{enumitem}
\setlist{nosep}
\theoremstyle{plain}
\newtheorem{theorem}{Theorem}[section]
\newtheorem{proposition}[theorem]{Proposition}

\newtheorem{corollary}[theorem]{Corollary}
\theoremstyle{definition}
\newtheorem{definition}[theorem]{Definition}
\newtheorem{assumption}[theorem]{Assumption}
\theoremstyle{remark}

\theoremstyle{plain}

\newtheoremrep{apxprop}[theorem]{Proposition}

\theoremstyle{definition}
\newtheorem{example}[theorem]{Example}

\newcommand{\reals}{\mathbb{R}}

\newcommand{\mc}{\mathcal}

\newcommand{\diag}{\operatorname{diag}}

\newcommand{\tr}{\operatorname{tr}}

\newcommand{\fw}{\texttt}
\newcommand{\randk}{\text{\texttt{rand}-$k$}}
\newcommand{\haark}{\text{\texttt{haar}-$k$}}
\newcommand{\normk}{\text{\texttt{norm}-$k$}}

\renewcommand*{\appendixprelim}{}

\allowdisplaybreaks[4]
\begin{document}






\title{Problem-dependent convergence bounds for randomized linear gradient compression}

\author{Thomas~Flynn,~
        Patrick~Johnstone,~
        and~Shinjae~Yoo~
\thanks{T. Flynn (\url{tflynn@bnl.gov}) and S. Yoo (\url{sjyoo@bnl.gov}) are  with the  Computing \& Data Sciences directorate, Brookhaven National Laboratory, Upton,
NY, 11973 USA.}
\thanks{P. Johnstone (\url{patrick.r.johnstone@gmail.com}) is with Meta, New York, NY, 10001 USA. Work performed while at Brookhaven National Laboratory.}
}

\markboth{Flynn \MakeLowercase{\textit{et al.}}: Problem-dependent convergence bounds for randomized linear gradient compression}{}

\maketitle

\begin{abstract}
In distributed optimization, the communication of model updates can be a performance bottleneck. Consequently, gradient compression has been proposed as a means of increasing optimization throughput. In general, due to information loss, compression introduces a penalty on the number of iterations needed to reach a solution. In this work, we investigate how the iteration penalty depends on the interaction between compression and problem structure, in the context of non-convex stochastic optimization. We focus on linear schemes, where compression and decompression can be modeled as multiplication with a random matrix. We consider several distributions of matrices, among them Haar-distributed  orthogonal matrices and matrices with random Gaussian entries. We find that the impact of compression on convergence can be quantified in terms of a smoothness matrix associated with the objective function, using a norm defined by the compression scheme. The analysis reveals that in certain cases, compression performance is related to low-rank structure or other spectral properties of the problem and our bounds predict that the penalty introduced by compression is significantly reduced compared to worst-case bounds that only consider the compression level, ignoring problem data. We verify the theoretical findings experimentally, including fine-tuning an image classification model.
\end{abstract}
\vspace{-1em}
\section{Introduction}

Training machine learning models on large datasets often involves distributed optimization. In this context distributed training means using multiple compute nodes to train a machine learning model collaboratively.
A significant performance bottleneck in distributed training comes from communication among training processes. In algorithms like Parallel SGD, model updates are synchronized at every optimization step, using a collective communication primitive like \emph{all-reduce} \cite{distbelief,zinkevich2010,paramserver}. In many HPC settings, there is an imbalance between local compute power and network bandwidth, such that time spent in network communication accounts for the majority of time in each iteration. 
One approach to address this bottleneck is to reduce communication by compressing model updates. In these schemes, model updates are compressed before communication, and these compressed representations are synchronized through collective communication and then decompressed. When the savings in communication time is greater than the overhead from compression and decompression, the result is faster training iterations. That is, we achieve higher throughput. However, higher throughput alone is not sufficient to achieve a benefit because practical compression schemes are lossy, introducing approximation errors into the training algorithm. 
Hence in order for compression to be useful, the increase in throughput needs to be of sufficient magnitude to make up for the increased number of training steps that may be required to compensate for these approximation errors.

In this work we examine the performance of several compression schemes. We analyze  how compression impacts the convergence of the algorithms, in terms of the number of iterations required to reach a solution. 
We focus our attention on three schemes in particular. In the first scheme, referred to as $\randk$, at each step of optimization a random subset of $k$ elements of the model update is synchronized. The second scheme, $\haark$, involves compressing gradients by projection to a random $k$-dimensional subspace, using random matrices distributed according to the Haar distribution on Orthogonal matrices. The third scheme we consider is  $\normk$, and involves compression using matrices with i.i.d. entries from the normal distribution $\mc{N}(0,1)$.
Our optimization objective is to find a local minimum of a function $f: \mathbb{R}^m\to\mathbb{R}$:
$$\min_{x\in\reals^m}f(x)$$
We allow $f$ to be non-convex, and our analysis bounds the number of iterations needed to find approximate stationary points. These are values of the parameter $x$ where $\|\nabla f(x)\|^2$ is small. The pseudo-code for the compression approach we study is listed in Algorithm \ref{algo1}. We assume that optimization is carried out by $r$ tasks. Iteration $t$ begins with a stochastic gradient computation at each task, producing the gradient estimates $g_{t}^i$. These gradients are then compressed locally, using a linear map $Q_t^T$ which is assumed to be the same across tasks. These compressed gradients are synchronized via all-reduce, and then decompressed using $Q_t$ into the vector $\Delta_t$. This $\Delta$ is then combined with a step size $\epsilon_t$ to update the parameter, obtaining $x_{t+1}$. Note that depending on the type of compression used, it is not necessary to actually form the compression matrix $Q_t$ (for instance, when using \randk{}).
\begin{algorithm}
  \caption{SGD with compression (runs on $r$ tasks)\label{algo1}}
\fw{\phantom{ }1:}\, \textbf{input:} Initial point $x_1 \in \reals^m$, same for all tasks. \\
\fw{\phantom{ }2:}\, \textbf{for} $t=1,2,\hdots$ \textbf{do} \\
\fw{\phantom{ }3:}\, \quad Compute stochastic gradient $g_t^i$ \\
\fw{\phantom{ }4:}\, \quad Compress: $\Delta_t^i \gets Q^T_tg_t^i$ \\
\fw{\phantom{ }5:}\, \quad All-reduce: $\Delta_t \gets \frac{1}{r}\sum_{i=1}^{r}\Delta_t^i$ \\
\fw{\phantom{ }6:}\,  \quad Decompress  $h_t \gets Q_t\Delta_t$ \\
\fw{\phantom{ }7:}\, \quad Perform update: $x_{t+1} = x_{t} - \epsilon_t h_{t}$
\end{algorithm}
Letting $g_{t} = \frac{1}{r}\sum_{i=1}^{r}g_{t}^i$ be the aggregate gradient estimate, the overall update to the iterate $x_{t+1}$ can be described as
\begin{equation}\label{q-formula}
x_{t+1} = x_{t} - \epsilon Q_{t}Q_{t}^Tg_{t}.
\end{equation}
\vspace{-3em}
\subsection{Related work}
A variety of compression schemes for optimization have been considered, including quantization  \cite{Alistarh2017}, such as single-bit quantization \cite{Seide2014Sep}, ternary quantization  \cite{Wen2017Dec}, random rounding   \cite{Horvoth2022sep}, along with techniques like sparsification \cite{Wangni2018Dec},  count sketches \cite{Ivkin2019Dec}, and low-rank projections as used in PowerSGD \cite{psgd}. Several schemes have associated convergence analyses, including those focusing on quantization \cite{Alistarh2017}, generic sparsity preserving compressors \cite{Khirirat2018Jun}, error-feedback and biased compressors \cite{sparsified2018} including the non-convex setting
    \cite{karimireddy19a}. 
For analysis, various characterizations of compression schemes have been used.
 For instance, \cite{efbv2022} defines an "$\omega$-unbiased compressor". Adapted to our setting of linear compressors (in the original definition of \cite{efbv2022}, the compressor need not be a linear mapping), an $\omega$-unbiased compressor is a distribution on $Q$ such that the following two properties hold, for any $g \in \reals^m$:
\begin{subequations}
\begin{align}
&\hspace{3em}\mathbb{E}[QQ^T]=I,  \label{unbiased}\\
&\mathbb{E}\left[\left\|g - QQ^Tg\right\|^2\right] \leq \omega\|g\|^2. \label{kcontraction}
\end{align}
\end{subequations}
Equation \eqref{unbiased} guarantees that the overall update in \eqref{q-formula} is unbiased, while \eqref{kcontraction} bounds the variance.
We note that the compressors studied in the present work, \normk{}, \randk{}, and \haark{} are all $\omega$-unbiased, with $\omega= \frac{m}{k}-1$ for $\randk{}$ and $\haark{}$ and  $\omega=\frac{m+1}{k}$ for  \normk{}.
 The compressor $\normk{}$ and $\randk{}$ also satisfy, after appropriate scaling, the criteria of $k$-contractions, as defined in \cite{sparsified2018}. 
While these abstractions facilitate analysis, they leave open the question of whether the performance bounds can be improved by taking into account the interaction between the compressor and the objective function. As can be seen below in Example \ref{ex:logreg}, the bounds that we derive can lead to tighter performance estimates compared to using \eqref{kcontraction} alone. While we focus on unbiased compressors, a systematic approach to biased compression appears in \cite{Beznosikov2023}.

 Using \textit{matrix smoothness} to analyze  compressors was introduced in \cite{Safaryan2021Dec} and further studied in \cite{Wang2022Dec,li2024detcgd}. For  a positive semidefinite matrix 
   $\mathbb{L}$, 
   let $\|\cdot\|_{\mathbb{L}}$ denote the  seminorm $\|x\|_{\mathbb{L}}\,=\,\sqrt{x^T\mathbb{L} x}$.
 The function $f$ is $\mathbb{L}$-smooth if 
 $\forall x,y$, 
 \begin{equation}\label{smoothness}
 f(x) \leq f(y) + \nabla f(y)^T(x-y) + \frac{1}{2}\|x-y\|_{\mathbb{L}}^2.
 \end{equation}
Intuitively, this matrix $\mathbb{L}$ can be taken to be a matrix that majorizes the Hessian matrix of the objective function. We adopt this formulation in our analysis.  An important difference in the present work is that we are not dealing with matrix step-sizes; rather, the matrix smoothness concept is used to analyze the interaction between the compressor and the problem structure in the scalar step-size setting. Furthermore, in this work we specifically investigate the performance of the compressors $\haark$ and $\normk$.

A notable compression technique using low-rank projections is PowerSGD \cite{psgd}. In PowerSGD, the initial values of $Q$ are random matrices, and these matrices adapt to the observed gradients, using subspace iteration. By comparison, in this work we consider the sequence of projection matrices to be i.i.d.  Follow up work such as  \cite{Ramesh2021Jul} demonstrates the utility of PowerSGD for large scale training. 

To address the loss in accuracy of the gradients associated with compression, error-feedback involves tracking the difference between the local gradient and it's compressed counterpart \cite{Seide2014Sep}\cite{strom15interspeech}.
The convergence properties of error-feedback schemes were considered in \cite{sparsified2018} and in \cite{karimireddy19a} for the case of non-convex functions. See also \cite{efbv2022} for a unified analysis that can handle variance reduction and error-feedback. In this work we focus on a simpler setting that does not maintain any state between iterations, in order to focus analysis on the statistical properties of the compressors.

Lower-bounds on the performance (overall communication, or number of iterations) of compressed training schemes have been considered in \cite{Albasyoni2020oct,Korhonen2021Dec}, which focus on the fundamental trade off between compression accuracy and compression level.
Notable work focusing on constrained optimization proved that the geometry of the constraint set impacts the overall iteration penalty of compression
\cite{Kasiviswanathan2021Dec}. Compression for distributed learning has been specialized for a variety of settings, including differential privacy  using quantization \cite{Agarwal2018Dec}, accounting for heterogeneous data across nodes \cite{Stich2020Sep}, compression in federated learning (FL) \cite{Li2022Dec,
Rothchild2020Jul}, decentralized training \cite{Vogels2020, Zhao2022Dec} and variations of adaptive gradient methods  \cite{Tang2021Jul}. 
Additionally, schemes that adjust compression based on gradient statistics \cite{Agarwal2021Mar} 
have been investigated.


In the compression scheme \haark{}, the matrix $Q_t$ is constructed using the Haar measure on orthogonal matrices. The Haar measure can be viewed as the natural generalization of the uniform distribution on the unit sphere to matrices. A defining property of this distribution is that if $Q$ is Haar-distributed and $U$ is any orthogonal matrix, then $UQ$ and $QU$ are Haar-distributed as well \cite{meckes_2019}. That is, for any measurable function $h:\mathbb{R}^{n \times n} \to \reals$, we have $\mathbb{E}[h(UQ)] = \mathbb{E}[h(QU)] = \mathbb{E}[h(Q)].$ In the compression method \haark{}, each $Q_{t}$ is made up of the first $k$ columns of an $n\times n$ Haar-distributed matrix (after scaling to ensure $\mathbb{E}[QQ^T] = I$).  An intuitive algorithm for sampling from the Haar distribution is to fill a matrix with i.i.d. normal random variables and then apply the Gram-Schmidt orthogonalization procedure; we use a more efficient algorithm based on the QR decomposition, following \cite{mezzadri2007generate}.

A result of our work is that random orthogonal compression leads to improved convergence bounds compared to Gaussian matrices. This agrees with signal reconstruction results \cite{irolasso2015}. That work
    analyzes a signal reconstruction when the observation is the product of a random matrix and the signal, considering both Gaussian and random orthogonal matrices for the observation matrix. Analysis shows that reconstruction error is better under orthogonal matrices.

\vspace{-1em}
\subsection{Our contributions}
This article has three main contributions. First,
    we derive optimization bounds for \haark{}, \normk{}, and \randk{}   that reflect how the compressor interacts with problem structure, leading to bounds that can be better than those solely accounting for compression while ignoring problem data.  To our knowledge, this is the first work to derive  performance bounds for compression with Haar random matrices.
    Second, the analytical formulas for the $\mc{Q}$-norm for different schemes (\randk{}, \haark{}, \normk{}) can  give insight into how the compressors compare against each other.
    Finally, experiments with linear regression and logistic regression suggest the theory captures real differences in performance among compressors.
\vspace{-1em}
\section{Properties of the compressors}
First, we define the three compression matrices considered.
\begin{definition}\label{def:mat} For $m\geq 1$ and $1\leq k \leq m$ we define three distributions on matrices $ Q \in \reals^{m\times k}$: 
\begin{itemize}
\item \haark{}: The $Q_t$ are such that $\sqrt{\frac{k}{m}}Q_t$ is distributed according to the Haar measure on $m\times k$ matrices \cite{meckes_2019}.

\item 
\normk{}: Each entry of 
$Q_t$ is an independent normal random variable with standard deviation $\frac{1}{\sqrt{k}}$.
\item \randk{}: The $Q_t$ are such that 
$\sqrt{\frac{k}{m}}Q_t$ has the distribution resulting from the removal of $m-k$ random columns from the identity matrix $I_{m\times m}$.
\end{itemize}
\end{definition}
\noindent Note that the scaling factors in the above definitions guarantee that the resulting estimators are unbiased when used for optimization.
We obtain different convergence bounds depending on the type of compression used. 
and in each case the impact of compression is quantified by the \textit{$\mc{Q}$-seminorm}:
\begin{definition}
Given a distribution $\mc{Q}$ over $\reals^{m\times k}$, the \textbf{$\boldsymbol{\mc{Q}}$-seminorm} is defined for  matrices $A \in \reals^{m\times m}$ as 
\begin{equation}\label{q-seminorm}
\|A\|_{\mc{Q}} = \left\|\mathbb{E}_{Q \sim \mc{Q}}\left[QQ^T A QQ^T\right]\right\|.
\end{equation}
\end{definition}
\noindent The $\mc{Q}$-seminorm plays a key role in our analyses below. When specialized to the case if linear regression, our convergence result depends on the $\mc{Q}$-seminorm of the data covariance matrix. In the general case of non-convex objectives, the convergence depends, roughly speaking, on a uniform bound on the $\mc{Q}$-seminorm of the Hessian matrix of the objective. In essence, the $\mc{Q}$-seminorm replaces the Lipschitz constant of the gradient in the analyses, and smaller values imply faster convergence.

\begin{figure}[t]
\centering
\includegraphics[width=\columnwidth]{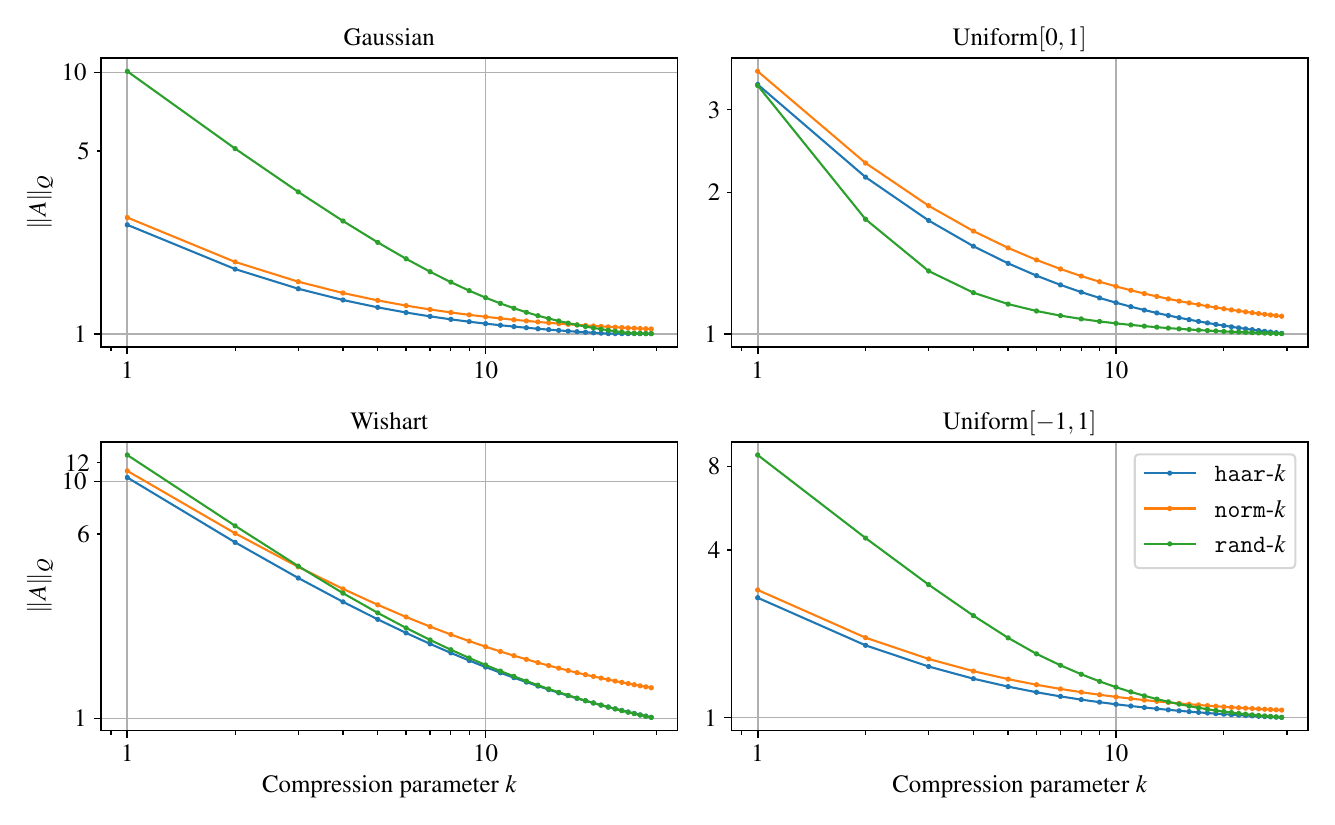}
\caption{Empirical values of the $\mc{Q}$-seminorm for matrices from several distributions, as the compression parameter varies. Upper left: Symmetric matrix with $\mc{N}(0,1)$ entries. Upper right: Symmetric matrix with $\mathrm{Unif}[0,1]$ entries. Lower left: Rank $r=30$. Wishart matrix. Lower right: Symmetric matrix with $\mathrm{Unif}[-1,1]$-entries. All matrices are $30 \times 30$.\label{fig-compress}}
\end{figure}

We now state our results on the 
$\mc{Q}$-seminorm for the various compression schemes we consider. In each case, we provide an exact formula for the 
inner expectation. 
\begin{proposition}\label{prop:q} Let $A$ be a matrix in $\reals^{m\times m}$.
Define the constant $\beta$ as
$\beta = \frac{m(m-k)}{(m+2)(m-1)}.$
Then $\|A\|_{\mc{Q}}$ for $Q \in \{\haark{},\randk{},\normk{}\}$ may be computed as follows:
\begin{align}
\|A\|_{\haark{}} &=
\frac{m}{k}\left\| (1-\beta) A + \beta \frac{\tr A}{m} I\right\|,\label{qsemi-orth-k} \\
\|A\|_{\randk{}} &=
\frac{m}{k}\left\| \frac{k-1}{m-1} A + \frac{m-k}{m-1}\diag A\right\|,\label{qsemi-rand-k}  \\
\|A\|_{\normk{}} 
&=
\frac{m}{k}\left\|\frac{k+1}{m}A + \frac{\tr A}{m} I\right\|. \label{qsemi-norm-k}
\end{align}
\end{proposition}
The derivations for these formulas may be found in the Appendix; see  Propositions \ref{proof-rand-k}, \ref{proof-orth-k},  and
 \ref{proof-norm-k} for $\randk{}, \haark{}$ and $\normk{}$ respectively. 
An empirical comparison of these bounds for several classes of matrices is shown in Figure \ref{fig-compress}. Each plot in the figure indicates the $\mathbb{Q}$-seminorm of a matrix drawn from a specific distribution, for the three values of $\mc{Q}$ and increasing values of $k$. In general, the \haark{} norm of the matrices is smallest, followed closely by \normk{}. We see that the relative performance of \randk{} depends heavily on the type of matrix, and in some cases on the compression level. We note that it is also possible to generate matrices where $\|\cdot\|_{\text{\normk{}}} \leq \|\cdot\|_{\text{\haark{}}}$, although this is not possible for the smallest values $k$ - for instance, when $k=1$, we can see that $\|\cdot\|_{\text{\haark{}}} = \frac{m}{m+2}\|\cdot\|_{\text{\normk{}}}$. 
Finally, it can be shown that, at least for 
$\mc{Q} \in \{ \haark{}, \normk{}, \randk{}\}$, 
the seminorm 
$\|\cdot\|_{\mc{Q}}$ 
is indeed a norm. See Proposition \ref{q-is-norm} for details.
\vspace{-1em}
\subsection{Relation between $\mc{Q}$-norm and other characterizations}
 \subsubsection{$\omega$-unbiased compressors}
The $\omega$-unbiased criterion can be framed in terms of the $\mc{Q}$-seminorm as follows:
\begin{proposition}\label{omega-q}
Suppose that $\mc{Q}$ is a distribution on matrices such that $\mathbb{E}[QQ^T] = I$. Then $\mc{Q}$ determines an $\omega$-unbiased compressor according to \ref{kcontraction} iff $\|I\|_{\mc{Q}} \leq (\omega+1)$.
\end{proposition}
\begin{proof}
By definition, a distribution $\mc{Q}$ determines a $\omega$-unbiased compressor iff for any $g$
we have $\mathbb{E}[\|g - QQ^Tg\|^2] \leq \omega\|g\|^2$.
Expanding out the left-hand side of this equation, and using the unbiasedness property, this is equivalent to 
$
\mathbb{E}[g^TQQ^TQQ^Tg] \leq (\omega+1)\|g\|^2
$.
Since $g$ is arbitrary, this equation is equivalent to stating $\|\mathbb{E}[QQ^TQQ^T]\| \leq (\omega+1)$
which is exactly the condition for boundedness of $\|I\|_{\mc{Q}}$.
\end{proof}
\noindent It can also be shown that a $\omega$-unbiased compressor satisfies $\|A\|_{\mc{Q}} \leq (\omega+1)\|A\|$. However, as shown in the Example, the benefit of the $\mc{Q}$ norm is that by taking into account more problem information it may lead to tighter optimization bounds.

\subsubsection{The quantity $\mc{L}(C,\mathbb{L})$ \label{sect:relate}}
     Another approach to capturing the interaction between a randomized compressor $C$ and the function $f$ is via the quantity $\mathcal{L}(C,\mathbb{L})$ defined in \cite{Safaryan2021Dec,Wang2022Dec}:
$  \mathcal{L}(C,\mathbb{L})\hspace{-0.2em}=\hspace{-0.2em} \inf\{\omega \geq 0 : \forall x, \mathbb{E}\left[\|C(x) - x\|_{\mathbb{L}}^2\right] \leq 
  \omega\|x\|^2\}.$
This  quantity captures the interaction between compression and smoothness matrices, and is a dominant factor in the convergence bounds in \cite{Safaryan2021Dec,Wang2022Dec}. When the underlying compressor $C$ is one of $\haark$ or $\normk$, our the results for the $\mc{Q}$-norm  (Propositions \ref{proof-orth-k}  and \ref{proof-norm-k}) could be applied to compute $\mc{L}(C,\mathbb{L})$. We leave  exploration of the relation between the $\mc{Q}$-norm and matrix step-sizes to future work.

\vspace{-1.0em}
\section{Optimization analysis}
We assume compression and decompression is unbiased:
\begin{assumption}\label{asu:unbiased}
Let $\mc{Q}$ be a distribution on matrices $Q\in\reals^{m\times k}$ such that $\mathbb{E}_{Q\sim\mc{Q}}\left[QQ^T\right] = I_{m\times m}.$
\end{assumption}
\noindent The primary assumption on the objective
$f:\reals^n \to \reals$ follows:
\begin{assumption}\label{asu:q-smooth}
The function $f : \mathbb{R}^n \to \reals$ is differentiable, bounded from below by $f^*$, and, for some $\mathbb{L} \succeq 0$, satisfies the $\mathbb{L}$-smoothness criteria \eqref{smoothness}. The constant $P\geq 0$ is an upper bound on the $\mc{Q}$-norm of  $\mathbb{L}$. That is,
 $\left\|\mathbb{L}\right\|_{\mc{Q}} \leq P.$
\end{assumption}
\noindent If a function is $L$-smooth (i.e., has  an $L$-Lipschitz continuous gradient), we can take $P=\|I\|_{\mc{Q}}L$. In some cases a sharper bound on $P$ is possible by considering the distribution $\mc{Q}$ in the context of the problem structure; see Example \eqref{ex:logreg} for the case of logistic regression.
\begin{figure}[t]
\centering
\includegraphics[width=0.5\columnwidth]{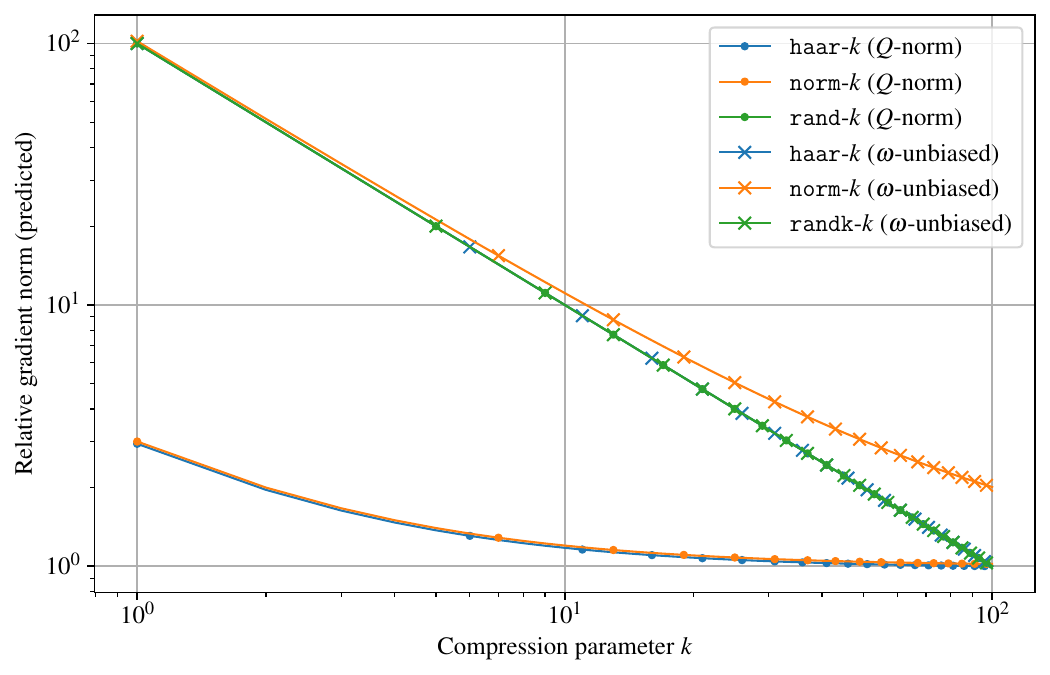}
\caption{Predicted convergence penalty vs compression parameter, for analyses based on the $\mc{Q}$-norm, and a problem-independent approach, in the case of a binary logistic regression problem. \label{fig:boundcompare}\vspace{-2em}}
\end{figure} The final assumption asserts that the stochastic gradients are unbiased with  bounded variance.
\begin{assumption}\label{asu:var}
There is a $\sigma\geq 0$ such that  $\mathbb{E}[g_t - \nabla f(x_t)] =0$ and $\mathbb{E}\left[\|\nabla f(x_t) - g_t\|^2\right] \leq \sigma^2$ for all $t\geq 1$.
\end{assumption}
\noindent We  now state our result on compression performance.
\begin{proposition}\label{prop:noncvx}
Let Assumptions \ref{asu:unbiased}, \ref{asu:q-smooth} and \ref{asu:var} hold. Let $\epsilon$ be 
$\epsilon = \min\{\frac{1}{P}, \frac{1}{\sigma\sqrt{N}}\}.$
Then the iterates 
of Algorithm \eqref{algo1} satisfy
$\frac{1}{N}\sum_{t=1}^{N}\mathbb{E}[ \left\|\nabla f(x_t)\right\|^2] 
\leq   \frac{2DP}{N} + \left(D + \frac{P}{2}\right)\frac{2\sigma}{\sqrt{N}}.$
\end{proposition}
\begin{proof}
Using \eqref{q-formula} and $\mathbb{L}$-smoothness, we have
\begin{align*} 
\mathbb{E}[f(x_{t+1})] \leq f(x_t) &- \epsilon \mathbb{E}[\nabla f(x_t)^T Q_tQ_t^T g_t] \\&+ \frac{\epsilon^2}{2}\mathbb{E}[g_t^T Q_tQ_t^T \mathbb{L} Q_tQ_t^Tg_t] 
\end{align*}
Writing $g_t=\nabla f(x_t) +\delta_t$, where $\delta_t = g_t - \nabla f(x_t)$ is a mean-zero error term, and using the bound on $\|\mathbb{L}\|_{\mc{Q}}$,
\begin{align*}
\mathbb{E}[g_t^T Q_tQ_t^T \mathbb{L} Q_tQ_t^Tg_t] 
&=
\mathbb{E}[\nabla f(x_t)^TQ_tQ_t^T\mathbb{L}Q_tQ_t^T\nabla f(x_t) ] \\&\quad+ \mathbb{E}[\delta_t Q_tQ_t^T\mathbb{L}Q_tQ_t^T\delta_t] \\
&\leq \mathbb{E}[\|\nabla f(x_t)\|^2]P +\sigma^2 P.
\end{align*}
Combining the above with Assumption \ref{asu:unbiased},
\begin{align*}
\mathbb{E}[f(x_{t+1})] \leq \mathbb{E}[f(x_t)]\hspace{-0.2em} -\hspace{-0.1em} \epsilon( 1 - \frac{\epsilon}{2}P)\,\mathbb{E}[\|\nabla f(x_t)\|^2] +\hspace{-0.2em}\frac{\epsilon^2}{2}\sigma^2 P
\end{align*}
From here we proceed as in \cite{ghadimi-lan}.
Rearranging and simplifying,
\begin{align*}
 \epsilon (
 1 -\frac{\epsilon}{2}P  
 )\,\mathbb{E}[\|\nabla f(x_t)\|^2]   &\leq \mathbb{E}[f(x_t)] -\mathbb{E}[f(x_{t+1})] + 
 \frac{\epsilon^2}{2}\sigma^2 P
\end{align*}
Summing from $t=1\hdots N$,  and dividing by $N$ gives
\begin{align*}
\frac{1}{N}\sum\limits_{t=1}^{N}
 \epsilon (
  1 - \frac{\epsilon}{2}P 
 )\,\mathbb{E}[\|\nabla f(x_{t})\|^2]   &\leq \frac{D}{N}  + 
\frac{\epsilon^2}{2}\sigma^2 P
\end{align*}
where $D=f(x_1)-f^*$. 
Noting that
$\frac{1}{\epsilon} 
\leq P + \sigma\sqrt{N},$
we obtain
\begin{align*}
&\frac{1}{N}\sum\limits_{t=1}^{N}\mathbb{E}[\|\nabla f(x_t)\|^2]
\leq \frac{D}{N\epsilon}\frac{1}{1-\epsilon P/2}  + 
 \frac{P}{2}\sigma^2\epsilon\frac{1}{1-\epsilon P/2}  \\
 &\quad\leq
  \frac{D}{N}\frac{1}{1- \epsilon P /2 }(P+ \sigma\sqrt{N})  + 
 \frac{P}{2}\sigma^2\frac{1}{\sigma\sqrt{N}}\frac{1}{1-\epsilon P /2}  
   \\
  &\quad\leq
   \frac{1}{N}\left(
2 DP
 \right)
 +
 \frac{2\sigma}{\sqrt{N}}(
  D+ \frac{P}{2}) . \quad \qedhere
\end{align*}
\end{proof}
\noindent 
Asymptotically, this represents a 
$\mc{O}(\frac{P}{N} + \frac{\sigma P}{\sqrt{N}})$ convergence rate, which is consistent with existing upper-bounds on non-convex SGD \cite{ghadimi-lan}, but with the $\|\mathbb{L}\|_{\mc{Q}}$-smoothness constant replacing the usual gradient smoothness term. As $\sigma$ tends to zero in Prop.  \ref{prop:noncvx}, we recover a result for exact gradient descent with compression, as presented below in Corollary \ref{prop:non-convex-no-noise}. 

\begin{corollary}\label{prop:non-convex-no-noise}
Let Assumptions \ref{asu:unbiased}, \ref{asu:q-smooth} and \ref{asu:var} hold with $\sigma=0$.
Let $\epsilon = \frac{1}{P}$.
Then 
$
\frac{1}{N}
\sum_{t=1}^{N}\mathbb{E}[\|\nabla f(x_t)\|^2] \leq 
\frac{2DP}{N}.
$
\end{corollary}
\noindent We use this corollary to illustrate some features of the bounds one can obtain in the example of binary logistic regression.
\begin{example}[Binary logistic regression]\label{ex:logreg}
Here, $f:\reals^m\to\reals$ is $f(x) = \frac{1}{B}{\textstyle\sum}_{i=1}^{B}f_i(x)$, where each $f_i$ is defined as follows:
$f_i(x) =  h(y_i x^T z_i)$
where $h$ is the the log softmax function $h(w) = \log\left(1+e^{-w}\right)$, $z_i$ is a feature vector and $y_i \in \{-1,1\}$ is the target label for the $i$th example. For the purposes of this example, we assume $\|z_i\|_2 \leq 1$. The function $f$ is bounded from below by zero and we may take $D = f(x_1)$. 
Noting that $|h''(w)| \leq \frac{1}{4}$ for all $w$, it can be shown (Lemma 1 in \cite{Safaryan2021Dec}) that $f$ is $\mathbb{L}$-smooth, where $\mathbb{L}$ is the matrix $\mathbb{L} = \frac{1}{4}\frac{1}{B}\sum_{i=1}^{B}z_i z_i^T.$
One can use properties of $\mathbb{L}$ to potentially bound $\|\mathbb{L}\|_{\mc{Q}}$ and thereby obtain bounds on optimization performance. In this example we demonstrate one such derivation to illustrate the ability of this approach to differentiate among the compressors and incorporate problem information. Primarily, we use that $\|\mathbb{L}\| \leq \frac{1}{4}$, $\tr \mathbb{L} = \frac{1}{4}$ and $\|\diag \mathbb{L}\| \leq \frac{1}{4}$. Plugging these values into the formulas for the $\mc{Q}$-norms in Proposition \ref{prop:q}, we obtain the following upper bounds for $\|\mathbb{L}\|_{\mc{Q}}$:
\begin{align*}
\|\mathbb{L}\|_{\haark{}} &\leq 
\frac{m}{k}\left[ (1-\beta)\|\mathbb{L}\| + \beta \frac{\tr \mathbb{L}}{m}\right] 
\leq 
\frac{1}{4}\frac{m}{k}\frac{k+2}{m+2},\\
\|\mathbb{L}\|_{\randk{}} &\leq \frac{m}{k}\left[ \frac{k-1}{m-1}\|\mathbb{L}\| + \frac{m-k}{m-1}\max_{i}\mathbb{L}_{i,i}\right] \leq \frac{1}{4}\frac{m}{k},\\ 
\|\mathbb{L}\|_{\normk{}} &\leq  \frac{m}{k}\left[ \frac{k+1}{m}\|\mathbb{L}\|+ \frac{\tr \mathbb{L}}{m}\right] \leq \frac{1}{4}\frac{k+2}{k}.
\end{align*}
These estimates for $\|\mathbb{L}\|_{\mc{Q}}$ may in turn be combined with Corollary \ref{prop:non-convex-no-noise} to obtain bounds on optimization performance. These bounds are plotted in Figure \ref{fig:boundcompare}. They have been uniformly scaled by dividing the curves by the corresponding upper bound bound that Corollary \ref{prop:non-convex-no-noise} would provide in the absence of compression, which is $\frac{D}{2N}$. On the other hand, if no problem information is available aside from $\|\mathbb{L}\| \leq \frac{1}{4}$, we can obtain performance bounds by combining Corollary \ref{prop:non-convex-no-noise} with the the inequality $\|\mathbb{L}\|_{\mc{Q}} \leq \|I\|_{\mc{Q}}\|\mathbb{L}\|$, essentially ignoring the interaction of the problem and the compressor. For this case, we use that $\|I\|_{\haark{}} =\|I\|_{\randk{}} = \frac{m}{k}$, while $\|I\|_{\normk{}} = \frac{m+k+1}{k}$. Bounds derived from these inequalities are included in Fig. \ref{fig:boundcompare}, and correspond to the $\omega$-unbiased curves.

We highlight two observations from Figure \eqref{fig:boundcompare}. Firstly, for $\haark{}$, and $\normk{}$, the performance guarantee is significantly better when analyzed using the $\mc{Q}$-norm, especially at small values of $k$. The second is that the $\mc{Q}$-norm analysis provides predictions that are qualitatively different: Under the $\omega$-unbiased analysis, $\normk{}$ is predicted to be the worst performer and $\haark{}$ is predicted to be on-par with $\randk{}$, while in the $\mc{Q}$-norm analysis, both $\haark{}$ and $\normk{}$ are predicted to perform significantly better than $\randk{}$.
 \qed
\end{example}
\vspace{-1em}
\section{Experimental results}
We present experimental results on least squares and logistic regression problems. The aim of these experiments is to investigate whether the bounds derived above correlate with real-world performance. The experiments confirm the theory above: As a rule, $\randk$ is the worst performer among the three compression schemes, while $\normk$ and $\haark$ have similar performance, with $\haark$ having an advantage. 
\subsubsection{Linear regression}
 The task is to learn regression coefficients that map an $m=100$ dimensional vector to a scalar. The number of input/output pairs is $10$ and they are randomly sampled from a Gaussian distribution on $\mathbb{R}^{100+1}$. In these experiments we use Corollary \ref{prop:non-convex-no-noise} to compute the step-sizes that minimize the upper bound on performance. The results are presented in the top row of Figure  \ref{fig:results}. Each plot shows a comparison of non-compressed and compressed training at various values of $k$. 
On the top left, we plot the ratio of the (squared) gradient norm returned by compressed and non-compressed training after a fixed number of steps $N=40$. Due to the lossy nature of compression, non-compressed training finds a parameter with a smaller gradient  and the gap between compressed and non-compressed training narrows as the compression parameter $k$ is increased. In this case, we can exactly calculate the right-hand side of Corollary \ref{prop:non-convex-no-noise} for both compressed and non-compressed training, to obtain an estimate of the relative performance of each compression scheme. This is indicated by '$\times$', and we see that the predicted performance comparison is within the error bars of performance in each case. On the top right we compare the algorithms by examining the number of iterations needed by compressed and non-compressed optimization to find a point with a small gradient, defined as $\|\nabla f(x)\|^2 < 10^{-3}$. Specifically, we plot the ratio of the number of iterations needed with and without compression. We find that the relative performance of the compression schemes follows a similar pattern to what was observed with the relative gradient norm. For reference, non-compressed optimization took an average of $N=15$ iterations to meet this criteria.
In these plots, the error bars represent one standard deviation of performance over $500$ random datasets.

\begin{figure}[t]

\centering
\includegraphics[width=0.475\columnwidth]{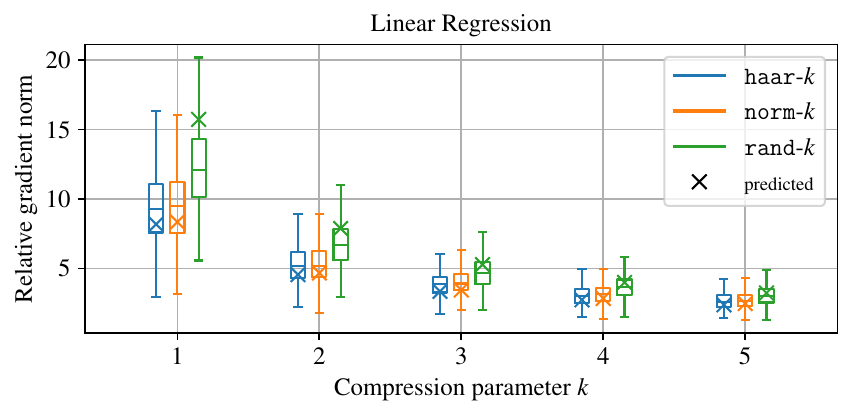}
\quad
\includegraphics[width=0.475\columnwidth]{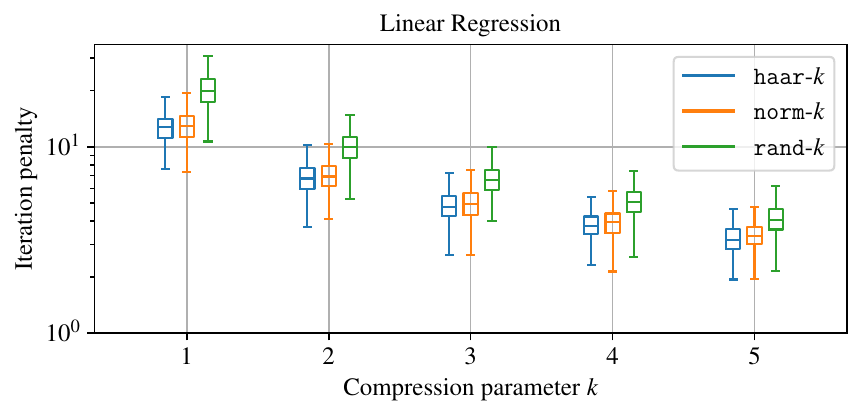}\\\includegraphics[width=0.475\columnwidth]{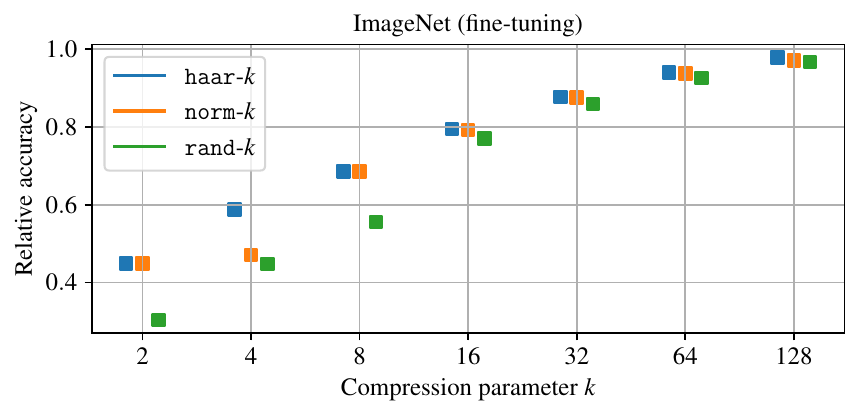}
\quad
\includegraphics[width=0.475 \columnwidth]{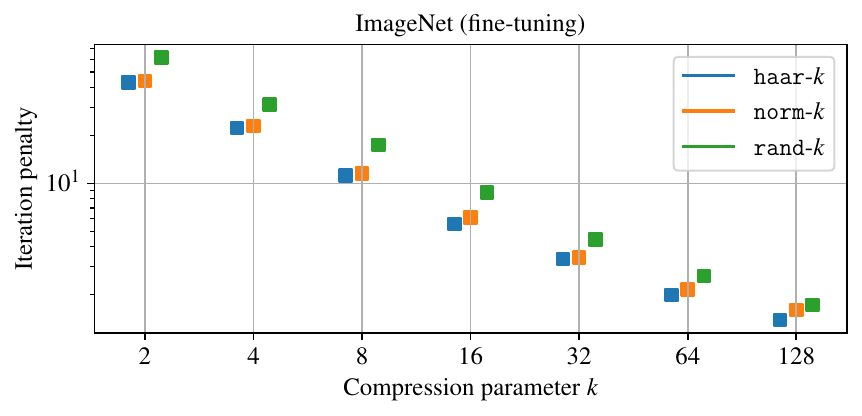}
\vspace{-1em}
\caption{Progress of optimization with and without compression.
The top plots indicate performance on a linear regression task, while the bottom plots show performance on a neural network training task. We omit error bars in the ImageNet experiments due to the low variability of the results.
\vspace{-1.5em}\label{fig:results}}
\end{figure}

\subsubsection{ImageNet fine-tuning}
As a second example, we considered fine-tuning the last layer of a ResNet-18 model for image classification \cite{he2016deep}. The initial model was pre-trained without compression on the ImageNet-21k dataset, after removing the 1000 classes that determine the ImageNet-1k dataset \cite{ILSVRC15}.  For fine-tuning, we take the pre-trained model and train the last layer to classify images from ImageNet-1k, using the standard training/validation split for that dataset. This fine-tuning step is equivalent to multinomial logistic regression, with input features determined by the activations at the penultimate layer of a ResNet model.  For each $k \in \{2^{i}\}_{1 \leq i \leq 7} $ and each compression type, including no compression, we identified the best learning rate via grid search in the set $\eta \in \{ 10^{-2 +3i/14}\}_{0 \leq i \leq 14}$. After the  learning rate is determined, we repeat each experiment three times. The batch size was $4096$, and hence our experiments model a scenario where $\ell$ nodes collaborate by providing gradients from a local minibatch of size $4096/\ell$. We plot two indicators of performance in the bottom row of Figure \ref{fig:results}. The bottom left compares the accuracy obtained by compressed and non-compressed training, for the various compressors and values of $k$, after running for a fixed number of iterations $N=840$. The $y$-axis represents the ratio of the accuracy of the compressed and non-compressed model. We see that at each $k$, the ranking of compression types agrees with what was observed for linear regression. 
The bottom right plot indicates the relative number of iterations needed to reach 50\% test accuracy by compressed and non-compressed optimization. That is, at each $k$ we indicate the iteration penalty for the compression schemes to obtain 50\% accuracy.\footnote{For reference, in these experiments an average of $N=470$ iterations were needed for non-compressed training to reach the target accuracy.}
As with linear regression, we find that $\haark{}$ has an advantage over $\normk{}$, while $\randk{}$ performs the worst out of the three.
\vspace{-0.7em}
\section{Conclusion}
In this paper, we explored the relationship between the performance of compressors and problem structure. We looked at random linear compression schemes and determined upper bounds on optimization performance in terms of problem data. These data included properties of the smoothness matrix $\mathbb{L}$, such as the norm and the trace. We believe that the results discussed here at not limited to random compression, but can be used to gain insight into more complex schemes such as PowerSGD and other approaches utilizing error-feedback.
\vspace{-0.7em}
\section*{Acknowledgement} 
This research was supported by the U.S. Department of Energy's Office of Science 
under Contract No. DE-SC0012704, and used resources of the National Energy Research Scientific Computing Center (NERSC), a Department of Energy Office of Science User Facility using NERSC award ASCR-ERCAP0027539.
The authors also acknowledge Roxana Geambasu (Columbia University) and Pierre Tholoniat (Columbia University) for helpful discussions. In particular, P. Tholoniat developed the initial model used in the fine-tuning experiments.
\vspace{-2em}
\bibliography{main}
\bibliographystyle{IEEEtran}
\appendices
\makeatletter
\renewcommand{\section}[1]{%
  \refstepcounter{section}%
}
\makeatother

\section{\label{apx:aux}}
\begin{proposition}\label{proof-rand-k}
Let $Q$ have the distribution resulting from the removal of $m-k$ random columns from the identity matrix $I_{m\times m}$.
(That is, $Q$ is defined like the \randk{} matrices from Definition \ref{def:mat}, without the scaling factor.)
Let $A$ be any matrix. Then
 $\|\mathbb{E}[QQ^TAQQ^T]\|  =
  \frac{k}{m}\left\|
\frac{k-1}{m-1}A+  \frac{m-k}{m-1}\diag A\right\|.$
\end{proposition}
\begin{proof}

The random matrix $QQ^T$ is made from taking the identity matrix and setting to zero all columns that are not in the chosen set of parameters to update. Let $I$ be the indices of columns that are non-zero.
We have $(QQ^T)_{i,i} = 1_{i\in I}$ and $\mathbb{E}[(QQ^T)_{i,i}] = p(i \in I)$ and
$(QQ^TAQQ)_{i,j}= 1_{i \in I, j \in I}A_{i,j}$.
Note that $p(i \in I) = \frac{k}{m}$. \footnote{We can show this by induction:
Let $p_{j}$ be the probability that $i$ is in a random interval of size $j$. Clearly $p_1= \frac{1}{m}$. For the induction, assume $p_{k-1} = \frac{k-1}{m}$. Then
$p_k = p_{k-1} + (1-p_{k-1})\frac{1}{m-(k-1)} 
= \frac{k}{m}$.} The probability $p( i \in I, j \in I)$ can be calculated according to the hypergeometric distribution:
\begin{align*}
p(i \in I\, , \,j \in I) &= \begin{cases}
\frac{k}{m} &\text{ if } i=j,  \\
\frac{k}{m}\frac{k-1}{m-1} &\text{ if } i \neq j.
\end{cases} 
\end{align*}
So $\mathbb{E}[QQ^T A QQ^T] = M\circ A$ where 
$M_{i,j} = \frac{k}{m}(\delta_{i,j} + (1-\delta_{i,j})\frac{k-1}{m-1}) =
\frac{k}{m}(\frac{k-1}{m-1} + (1-\frac{k-1}{m-1})\delta_{i,j})$. 
Expanding the definition of $M$, we see that 
$M\circ A = \frac{k}{m}\frac{k-1}{m-1}A +  \frac{k}{m}(1-\frac{k-1}{m-1})\diag A.$
Taking the norm gives our result.
\end{proof}

\begin{proposition}[\cite{meckes_2019}]\label{rand-orth-res}
    If $Q$ is distributed according to the Haar measure on $\reals^{m\times m}$ then    $\mathbb{E}[Q_{1,1}] = \mathbb{E}[Q_{1,1}Q_{2,1}] = 0$ and%
  \begin{subequations}
  \begin{align}
    \label{p2}
    \mathbb{E}[Q_{1,1}^2] &= \frac{1}{m},\\
    \label{p3}
    \mathbb{E}[Q_{1,1}^4]  &= \frac{3}{m(2+m)},\\
    \label{p4}
    \mathbb{E}[Q_{1,1}^2Q_{2,1}^2] &= \frac{1}{m(m+2)},\\
    \label{p5}
    \mathbb{E}[Q_{1,1}^2Q_{2,2}^2] &= \frac{m+1}{(m-1)m(m+2)},\\
    \label{p6}
    \mathbb{E}[Q_{1,1}Q_{1,2}Q_{2,1}Q_{2,2}] &= -\frac{1}{(m-1)m(m+2)}.
    \end{align}
    \end{subequations}
\end{proposition}
\begin{proof}
We present the derivation for \eqref{p5}  and \eqref{p6}, and refer the reader to \cite{meckes_2019} for the rest, and in particular  Proposition 2.5 for \eqref{p3} and \eqref{p4}.
By symmetry, the joint distribution of $(Q_{1.1}, Q_{2,2})$ is the same as $(Q_{1,1}, Q_{i,2})$ for any $i\neq 1$. Hence
\begin{align*}
\mathbb{E}[Q_{1,1}^2Q_{2,2}^2] &= 
\mathbb{E}\big[Q_{1,1}^2\frac{1}{m-1}\sum\limits_{j=2}^{m}Q_{j,2}^2\big] \\
&=
\mathbb{E}[Q_{1,1}^2\frac{1}{m-1}(1-Q^2_{1,2})] \\
&=
\frac{1}{m-1}\mathbb{E}[Q_{1,1}^2] - \frac{1}{m-1}\mathbb{E}[Q_{1,1}^2Q_{1,2}^2] 
\end{align*}
In the last step we used \ref{p2} and \ref{p4}. This shows \ref{p5}.
Finally, for \ref{p6} we start with the symmetry that the joint distribution of $Q_{1,1},Q_{1,2},Q_{i,1},Q_{i,2}$ is the same for all $i> 1$:
\begin{align*}
\mathbb{E}[Q_{1,1}Q_{1,2}Q_{2,1}Q_{2,2}] &= 
\frac{1}{m-1}\mathbb{E}\big[Q_{1,1}Q_{1,2}\sum\limits_{i=2}^{m}Q_{i,1}Q_{i,2}\big] \\
&=
\frac{1}{m-1}\mathbb{E}[Q_{1,1}Q_{1,2}(-Q_{1,1}Q_{1,2})] \\
&= -\frac{1}{m-1}\mathbb{E}[Q_{1,1}^2Q_{1,2}^2].
\end{align*}
For the second step above, orthogonality of the columns yields $\sum_{i=2}^{m}Q_{i,1}Q_{i,2} = -Q_{1,1}Q_{1,2}$. Finally,  apply \eqref{p4}.
\end{proof}

\begin{proposition}\label{proof-orth-k}
Let $A$ be an $m\times m$  symmetric matrix. Let $Q$ be sampled from the Haar measure on $m\times k$ matrices. Then
$
\mathbb{E}[QQ^TAQQ^T] 
= \frac{k}{m}(1-\beta) A + \left(\tr A \right)\frac{k}{m^2}\beta I
$
where $\beta$ is defined as in Proposition \ref{prop:q}.
\end{proposition}
\begin{proof}
Since the map $A \mapsto \mathbb{E}[QQ^T A QQ^T]$ is linear, and $A$ is symmetric, it suffices to prove the statement for rank-one matrices of the form $A=vv^T$ where $\|v\|=1$.
Observe that
\begin{equation}\label{q-eqn}
(QQ^TAQQ^T)_{i,j}\hspace{-0.3em}=\hspace{-0.3em}
\sum\limits_{r=1}^{m}\sum\limits_{a=1}^{m}A_{r,a}\hspace{-0.3em}\sum\limits_{d=1}^{k}\sum\limits_{c=1}^{k}
Q_{i,c}Q_{r,c}Q_{a,d}Q_{j,d} 
\end{equation}
For fixed $i,j,r,a$ the inner sum on the right has expectation
\begin{align*}
\sum\limits_{d=1}^{k}\sum\limits_{c=1}^{k}
\mathbb{E}[Q_{i,c}Q_{r,c}Q_{a,d}Q_{j,d}]
&=
k\mathbb{E}[Q_{i,1}Q_{r,1}Q_{a,1}Q_{j,1}] \\&\hspace{-2em}+
k(k-1)
\mathbb{E}[Q_{i,2}Q_{r,2}Q_{a,1}Q_{j,1}]. 
\end{align*}
Above we split up the sum based on the cases $c=d$ and $c\neq d$. Combining this with \eqref{q-eqn} gives
\begin{equation}\label{general-i-j}
\begin{split}
&\mathrlap{\mathbb{E}[(QQ^TAQQ^T)_{i,j}] = k\sum\limits_{r=1}^{m}\sum\limits_{a=1}^{m}A_{r,a}  \mathbb{E}[Q_{i,1}Q_{r,1}Q_{a,1}Q_{j,1}]} \\&\quad+
 k (k-1)\sum\limits_{r=1}^{m}\sum\limits_{a=1}^{m}A_{r,a}
\mathbb{E}[Q_{i,2}Q_{r,2}Q_{a,1}Q_{j,1}]. 
\end{split}
\end{equation}
Let $i\hspace{-0.2em}=\hspace{-0.2em}j$. 
The first term on the right is (ignoring the scalar $k$),
\begin{equation}\label{ii-first}
\begin{split}
&\sum\limits_{r=1}^{m}\sum\limits_{a=1}^{m}A_{r,a}  \mathbb{E}[Q_{i,1}^2Q_{r,1}Q_{a,1}] 
\\&=\sum\limits_{r=1}^{m}A_{r,r}  \mathbb{E}[Q_{i,1}^2Q_{r,1}^2]
+\sum\limits_{r=1}^{m}\sum\limits_{a\neq r}^{m}A_{r,a}  \mathbb{E}[Q_{i,1}^2Q_{r,1}Q_{a,1}] 
\end{split}
\end{equation}
The double sum on the right of \eqref{ii-first} vanishes since
\begin{equation}\label{dbl-ii}
\begin{aligned}
&\sum\limits_{r=1}^{m}\sum\limits_{a\neq r}^{m}A_{r,a}  \mathbb{E}[Q_{i,1}^2Q_{r,1}Q_{a,1}] \\
&= \sum\limits_{a\neq i}^{m}A_{i,a}  \mathbb{E}[Q_{i,1}^3Q_{a,1}] 
+ \sum\limits_{r\neq i}^{m}\sum\limits_{a\neq r}^{m}A_{r,a}  \mathbb{E}[Q_{i,1}^2Q_{r,1}Q_{a,1}] \\
&=
 \sum\limits_{r\neq i}^{m}\hspace{-0.25em}A_{r,i}  \mathbb{E}[Q_{i,1}^3Q_{r,1}] 
 + \sum\limits_{r\neq i}^{m}\hspace{-0.2em}\sum\limits_{a\neq r,a\neq i}^{m}\hspace{-1em}A_{r,a}  \mathbb{E}[Q_{i,1}^2Q_{r,1}Q_{a,1}]  \\
 &=
 \sum\limits_{r\neq i}^{m}\sum\limits_{a\neq r,a\neq i}^{m}A_{r,a}  \mathbb{E}[Q_{1,1}^2Q_{2,1}Q_{3,1}] = 0  
\end{aligned}
\end{equation}
For the first term on the right of \eqref{ii-first} we have, using \eqref{p3}, \eqref{p4},
\begin{equation}\label{ii-first-right}
\begin{split}
&\sum\limits_{r=1}^{m}\hspace{-0.25em}A_{r,r}  \mathbb{E}[Q_{i,1}^2Q_{r,1}^2]\hspace{-0.25em} =\hspace{-0.25em}
A_{i,i} \mathbb{E}[Q_{i,1}^4]\hspace{-0.15em}+\hspace{-0.15em}
\sum\limits_{r\neq i}^{m}\hspace{-0.25em}A_{r,r} \mathbb{E}[Q_{1,1}^2Q_{2,1}^2] \\
&=
 A_{i,i} \frac{3}{m(m+2)}+
\sum\limits_{r\neq i}^{m}A_{r,r} \frac{1}{m(m+2)}. 
\end{split}
\end{equation}
Continuing with $i=j$, the second term on the right of \eqref{general-i-j} is 
\begin{equation}\label{ii-second}
\begin{split}
&\sum\limits_{r=1}^{m}\sum\limits_{a=1}^{m}\hspace{-0.2em}A_{r,a} \mathbb{E}[Q_{i,2}Q_{r,2}Q_{a,1}Q_{i,1}] \hspace{-0.2em}=\hspace{-0.4em} 
\sum\limits_{r=1}^{m}\hspace{-0.2em}A_{r,r}\mathbb{E}[Q_{i,2}Q_{r,2}Q_{r,1}Q_{i,1}] \\&\quad+
\sum\limits_{r=1}^{m}\sum\limits_{a\neq r}^{m}A_{r,a}\mathbb{E}[Q_{i,2}Q_{r,2}Q_{a,1}Q_{i,1}] \\
&\quad=
\sum\limits_{r=1}^{m}A_{r,r}\mathbb{E}[Q_{i,2}Q_{r,2}Q_{r,1}Q_{i,1}] \\
&\quad=
A_{i,i} \frac{1}{m(m+2)} -
\sum\limits_{r\neq i}^{m}A_{r,r} \frac{1}{(m-1)m(m+2)}.
\end{split}
\end{equation}
Above, we used  \eqref{p4} and \eqref{p6}. Combining \eqref{general-i-j}, \eqref{ii-first}, \eqref{dbl-ii}, \eqref{ii-first-right} and \eqref{ii-second} we get that for $i=j$,
\begin{align*}
&\mathbb{E}[(QQ^TAQQ^T)_{i,i}] \\&= 
k A_{i,i} \frac{3}{m(m+2)}+
k\sum\limits_{r\neq i}^{m}A_{r,r} \frac{1}{m(2+m)} + \frac{k(k-1)}{m(m+2)}A_{i,i} \\&\quad-k(k-1)
\sum\limits_{r\neq i}^{m}A_{r,r} \frac{1}{m-1}\frac{1}{m(m+2)} \\
&=
\frac{k}{m}\frac{(k+2) }{(m+2)} A_{i,i}+
\frac{k}{m}\frac{(m-k)}{(m-1)(m+2)} \sum\limits_{r\neq i}^{m}A_{r,r}.
\end{align*}
Recall that $A=vv^T$ and let $\alpha = \frac{k+2}{m+2}$. Then
\begin{equation*}
\mathbb{E}[(QQ^TAQQ^T)_{i,i}]
=
\frac{k}{m}\alpha v_i^2 + \frac{k}{m^2}\beta\sum\limits_{r\neq i}v_{r}^2 
\end{equation*} 
Next, assume that $i\neq j$. For the first term on the right of \eqref{general-i-j},
\begin{equation}\label{i-j-first}
\begin{split}
&\sum\limits_{r=1}^{m}\sum\limits_{a=1}^{m}A_{r,a}\mathbb{E}[Q_{i,1}Q_{r,1}Q_{a,1}Q_{j,1}]
 \\&=
\sum\limits_{r=1}^{m}A_{r,r} \mathbb{E}[Q_{i,1}Q^2_{r,1}Q_{j,1}] +
\sum\limits_{r=1}^{m}\sum\limits_{a\neq r}^{m}A_{r,a}  \mathbb{E}[Q_{i,1}Q_{r,1}Q_{a,1}Q_{j,1}] \\
&=
\sum\limits_{a\neq i}^{m}A_{i,a}  \mathbb{E}[Q^2_{i,1}Q_{a,1}Q_{j,1}] 
+
\sum\limits_{r\neq i}^{m}\sum\limits_{a\neq r}^{m}A_{r,a} \mathbb{E}[Q_{i,1}Q_{r,1}Q_{a,1}Q_{j,1}] \\
&=
A_{i,j}  \mathbb{E}[Q^2_{i,1}Q^2_{j,1}] 
+
\sum\limits_{r\neq i}^{m}\sum\limits_{a\neq r}^{m}A_{r,a}  \mathbb{E}[Q_{i,1}Q_{r,1}Q_{a,1}Q_{j,1}] 
\end{split}
\end{equation}
The double sum on the right of the above satisfies
\begin{equation}\label{i-j-first-2}
\begin{aligned}
\sum\limits_{r\neq i}^{m}\sum\limits_{a\neq r}^{m}\hspace{-0.3em} A_{r,a}\mathbb{E}[Q_{i,1}Q_{r,1}Q_{a,1}Q_{j,1}] 
\hspace{-0.2em}&=\hspace{-0.3em}
\sum\limits_{a\neq j}^{m}\hspace{-0.3em}A_{j,a}  \mathbb{E}[Q_{i,1}Q_{a,1}Q^2_{j,1}] \\&\hspace{-6em}+
\sum\limits_{r\neq i,r\neq j}^{m}\sum\limits_{a\neq r}^{m}A_{r,a} \mathbb{E}[Q_{i,1}Q_{r,1}Q_{a,1}Q_{j,1}]
\\
&=
A_{j,i} \mathbb{E}[Q_{i,1}^2Q^2_{j,1}]
\end{aligned}
\end{equation}
Combining \eqref{i-j-first} and \eqref{i-j-first-2}, the first term on the right of \eqref{general-i-j} is
\begin{equation}\label{ijright}\sum\limits_{r=1}^{m}\sum\limits_{a=1}^{m}A_{r,a}\mathbb{E}[Q_{i,1}Q_{r,1}Q_{a,1}Q_{j,1}]=
A_{i,j}\frac{2}{m(m+2)},\end{equation}
due to \eqref{p4}. For the second term on the right of \ref{general-i-j} we have
\begin{equation}\label{ijright2}
\begin{aligned}
&\sum\limits_{r=1}^{m}\sum\limits_{a=1}^{m}A_{r,a}
\mathbb{E}[Q_{i,2}Q_{r,2}Q_{a,1}Q_{j,1}] =
\sum\limits_{a=1}^{m}A_{i,a}
\mathbb{E}[Q_{i,2}^2Q_{a,1}Q_{j,1}] \\&\quad\quad+
\sum\limits_{r\neq i}^{m}\sum\limits_{a=1}^{m}A_{r,a}
\mathbb{E}[Q_{i,2}Q_{r,2}Q_{a,1}Q_{j,1}] 
\end{aligned}
\end{equation}
The first term on the right of the above satisfies
\begin{equation}\label{ijright3}
\begin{split}
\sum\limits_{a=1}^{m}A_{i,a}
\mathbb{E}[Q_{i,2}^2Q_{a,1}Q_{j,1}] 
&=
A_{i,j}\mathbb{E}[Q_{i,2}^2Q_{j,1}^2] \\& +
\sum\limits_{a\neq j}^{m}A_{i,a}\mathbb{E}[Q_{i,2}^2Q_{a,1}Q_{j,1}] \\
&=
A_{i,j}\frac{m+1}{(m-1)m(m+2)}, 
\end{split}
\end{equation}
where we used \eqref{p5}. On the other hand, the second term satisfies
\begin{equation}\label{ijright4}
\begin{aligned}
&\sum\limits_{r\neq i}^{m}\sum\limits_{a=1}^{m}A_{r,a}\mathbb{E}[Q_{i,2}Q_{r,2}Q_{a,1}Q_{j,1}] \\
&=\hspace{-0.4em}
\sum\limits_{r\neq i}^{m}\hspace{-0.3em}A_{r,j}\mathbb{E}[Q_{i,2}Q_{r,2}Q^2_{j,1}] \hspace{-0.2em}+\hspace{-0.2em}
\sum\limits_{r\neq i}^{m}\sum\limits_{a\neq j}^{m}A_{r,a}\mathbb{E}[Q_{i,2}Q_{r,2}Q_{a,1}Q_{j,1}] \\
&=
\sum\limits_{a\neq j}^{m}A_{j,a}\mathbb{E}[Q_{i,2}Q_{j,2}Q_{a,1}Q_{j,1}] \\&\quad\quad+
\sum\limits_{r\neq i,r\neq j}^{m}\sum\limits_{a\neq j}^{m}A_{r,a}\mathbb{E}[Q_{i,2}Q_{r,2}Q_{a,1}Q_{j,1}] \\
&=A_{j,i}\mathbb{E}[Q_{i,2}Q_{j,2}Q_{i,1}Q_{j,1}] +
\hspace{-0.7em}\sum\limits_{a\neq j,a\neq i}^{m}\hspace{-0.7em}A_{j,a}\mathbb{E}[Q_{i,2}Q_{j,2}Q_{a,1}Q_{j,1}] \\
&=
 -
A_{j,i}\frac{1}{(m-1)m(m+2)}.
\end{aligned}
\end{equation}
Here we used  \eqref{p6}. Applying \eqref{ijright2},\eqref{ijright3} and \eqref{ijright4},
\begin{equation}\label{i-j-second}
\begin{split}
&\sum\limits_{r=1}^{m}\sum\limits_{a=1}^{m}A_{r,a}
\mathbb{E}[Q_{i,2}Q_{r,2}Q_{a,1}Q_{j,1}] \\&=
A_{i,j}\frac{m+1}{(m-1)m(m+2)} -
A_{j,i}\frac{1}{(m-1)m(m+2)}.
\end{split}
\end{equation}
Combining \eqref{general-i-j}, \eqref{ijright}, and \eqref{i-j-second} we see that when $i\neq j$,
\begin{align*}
&\mathbb{E}[(QQ^TAQQ^T)_{i,j}] \\
&\hspace{0em}=
\left(\frac{2k(m-1)+k(k-1)(m+1)-k(k-1)}{(m-1)m(m+2)}\right)v_{i}v_j  \\
&=
\frac{k}{m}
\left(1-\beta\right)v_{i}v_j . 
\end{align*}
As $\|v\|=1$ and $\alpha-\frac{\beta}{m} = 1-\beta$, we summarize the above as
\[
\mathbb{E}[(QQ^TAQQ^T)_{i,j}]
=
\small\left\{
\begin{aligned}
&\frac{k}{m}(1-\beta) v_i^2 + \frac{k}{m^2}\beta &\text{ if } i = j, \\
&\frac{k}{m}(1-\beta) v_i v_j & \text{ else. }   \qedhere
\end{aligned}
\right.
\]
\end{proof}
\begin{proposition}\label{proof-norm-k}
Let $Q$ be an $m\times k$ matrix whose entries are i.i.d samples from $\mc{N}(0,1)$. (That is, distributed like the \normk{} matrices from Definition  \ref{def:mat}, without the scaling factor.). Then
$\mathbb{E}[QQ^T AQQ^T] = k(k+1)A + k(\tr A) I$.
\end{proposition}
\begin{proof}
The proof is similar to that of Proposition \ref{proof-orth-k}, with simplifications due to the fact the entries of $Q$ are now i.i.d.\end{proof}

\begin{proposition}\label{q-is-norm}
Let $\mc{Q} \in \{ \haark{}, \normk{}, \randk{}\}$. Then $\|\cdot\|_{\mc{Q}}$ determines a norm on $m \times m$ matrices.
\end{proposition}
\begin{proof}
Let $\mc{Q}=\haark{}$ and suppose that $\|A\|_{\mc{Q}}=0$. Equation \eqref{qsemi-orth-k} implies
$A = -\frac{\beta}{1-\beta}\frac{\tr A}{m}I$. Hence $A$ is a scalar multiple of the identity, and taking traces we see $\tr A = - \frac{\beta}{1-\beta}\tr A$, implying $\tr A=0$. Since the diagonal elements of $A$ are all equal, we must have $A=0$. The proof for $\normk{}$ is almost the same.
Suppose that $\mc{Q}=\randk{}$ and $\|A\|_{\mc{Q}}=0$. According to \eqref{qsemi-rand-k}, we must have $A = -\alpha \diag A$, for $\alpha=\frac{m-k}{k-1} > 0$. Hence $A$ is diagonal and $\diag A = -\alpha  \diag A$, implying $A=0$.
\end{proof}
\renewcommand{\appendixsectionformat}[2]{%
  \noindent\vspace{-2.0em}%
}
\end{document}